\documentclass[12pt]{amsart}

\usepackage[utf8]{inputenc}
 
\title{The maximum likelihood degree of sparse polynomial systems} 
\date{\today}
 
\usepackage[pagebackref=true,colorlinks=true, linkcolor=black ,  citecolor=blue, urlcolor=black]{hyperref}
\usepackage{amsfonts,latexsym,wasysym, mathrsfs, mathtools,hhline,color, bm}
\usepackage[all, cmtip]{xy}

\usepackage{url}

\definecolor{hot}{RGB}{65,105,225}

\usepackage{ textcomp }
\usepackage{ tipa }

\usepackage{graphicx,enumerate}

\usepackage{enumitem}
\usepackage[normalem]{ulem}

 \usepackage{amsthm}  
 \usepackage{amsmath} 
 \usepackage{amsfonts}  
 \usepackage{amssymb}
 \usepackage{amscd} 
 \usepackage[all]{xy} 
\usepackage[dvipsnames]{xcolor}  
\usepackage{graphicx}
\usepackage{pdfsync}
\usepackage{url}

\newtheorem{thm}{Theorem}[section]

\theoremstyle{definition}
\newtheorem{ex}[thm]{Example}
\newtheorem{theorem}[thm]{Theorem}
\newtheorem{lem}[thm]{Lemma}
\newtheorem{prop}[thm]{Proposition}
\newtheorem{cor}[thm]{Corollary}
\theoremstyle{definition} 
\newtheorem{defn}[thm]{Definition}
\newtheorem{rem}[thm]{Remark}
\numberwithin{equation}{section}

\newcommand{\newt}{{\mathrm{Newt}}}
\newcommand{\val}{{\mathrm{val}}}
\newcommand{\init}{{\mathrm{init}}}
\newcommand{\conv}{{\mathrm{Conv}}}
\newcommand{\mvol}{{\mathrm{MVol}}}

\newcommand{\magenta}[1]{{\color{magenta}#1}}
\newcommand{\JIR}{\magenta}

\newcommand{\NN}{\mathbb{N}_{\geq 0}}
\newcommand{\CC}{\mathbb{C}}
\newcommand{\AS}{\mathcal{A}}

\newcommand{\defcolor}[1]{{\color{RoyalBlue}#1}} 
\newcommand{\demph}[1]{\defcolor{{\sl #1}}}

\DeclareMathOperator{\Vertex}{Vert}
\DeclareMathOperator{\SL}{SL}

\newcommand{\lagObjective}{\Lambda}

\begin{document}

\author[Lindberg]{Julia Lindberg}
\address{
Julia Lindberg \\
Department of Electrical and Computer Engineering \\
University of Wisconsin-Madison \\ USA
}
\email{jrlindberg@wisc.edu}
\urladdr{\url{https://sites.google.com/view/julialindberg/home}}

\author[Nicholson]{Nathan Nicholson}
\address{Nathan Nicholson \\
Department of Math\\
University of Wisconsin-Madison \\ USA}
\email{nlnicholson@wisc.edu}
\urladdr{\url{https://math.wisc.edu/graduate-students/}}

\author[Rodriguez]{Jose Israel Rodriguez}
\address{
Jose Israel Rodriguez\\
Department of Math\\
University of Wisconsin-Madison \\
 USA
} 
\email{jose@math.wisc.edu}
\urladdr{\url{https://people.math.wisc.edu/~jose/}}

\author[Wang]{Zinan Wang}
\address{
Zinan Wang  \\
Department of Math\\
University of Wisconsin-Madison \\
USA
}
\email{zwang894@math.wisc.edu}
\urladdr{\url{https://sites.google.com/wisc.edu/zinanwang/}}

\begin{abstract}

    We consider statistical models arising from the common set of solutions to a sparse polynomial system with general coefficients. The maximum likelihood degree counts the number of critical points of the likelihood function restricted to the model. 
    We prove the maximum likelihood degree of a sparse polynomial system is determined by its Newton polytopes and
    equals the mixed volume of a related Lagrange system of~equations.


\end{abstract}

\maketitle

\section{Introduction}\label{s:introduction}
Maximum likelihood estimation is a statistical method of density estimation that seeks to maximize the probability that a given set of samples comes from a distribution. Given 
independent and identically distributed (iid)
 samples $s^{(1)},\ldots, s^{(N)}$ we can form a \demph{data vector} 
$u \in \triangle_{n-1} := \{p \in \mathbb{R}^n_{> 0} : \sum_{i=1}^n p_i = 1 \}$ 
 which counts the fraction of times each event happened in the sample set $s^{(1)},\ldots, s^{(N)}$.

Given $u$, the \demph{log likelihood function} for a discrete random variable is given by 
\[ \log(p_1^{u_1} \cdots p_n^{u_n}) = u_1 \log(p_1) + \ldots + u_n \log(p_n). \ \]
Maximum likelihood estimation aims 
to select the set of points $p \in \triangle_{n-1}$ that maximizes the likelihood that $u$ came from that distribution. In many instances, we assume that our density $p$ lives in a \demph{statistical model} $\mathcal{M} \subseteq \triangle_{n-1}$. In this setup, maximum likelihood estimation amounts to solving the (often) nonconvex optimization~problem 
\[  \max_p \ u_1 \log(p_1) + \ldots + u_n \log(p_n) \quad \text{subject to} \quad p \in \mathcal{M}. \]

This is the primary problem under consideration. While nonconvex optimization is often much more challenging than its convex counterpart, methods exist to tackle this problem. We consider the setup where $\mathcal{M}$ is defined by a set of polynomial equations and use tools from algebraic geometry to study the critical points of this optimization problem. 
This problem has been studied from several points of view.
An algebraic geometry approach and definition of maximum likelihood (ML) degree 
was made in \cite{MR2230921, HostenSolving}. 
The results  in \cite{MR3103064} 
show that the ML degree of a smooth variety equals a signed Euler characteristic, and
in the case of a hypersurface, that the ML degree equals a signed volume of a Newton polytope.
For the singular case, formulas for the ML Degree are given by the Euler obstruction function~\cite{MR3686780}. 
ML degrees also make an appearance in toric geometry \cite{MR3907355, MR4103774} and are studied for other statistical models~\cite{AKRS2021,DM2021,MR2988436,MR4219257,MR4196404}.

Specifically, we consider when $\mathcal{M}$ is given by the variety of a system of \demph{sparse polynomial equations}. 
Sparse polynomials have been studied in several 
contexts in numerical algebraic geometry~\cite{ huber1995a,VVC}.
A good introduction to this material is \cite[Chapter 3]{Sturmfels-CBMS}. 
Following similar conventions as those in \cite{SDSS}, 
we specify a family of sparse polynomials by its monomial support using the following notation.
For  each 
$\alpha = (\alpha_1,...,\alpha_n)\in \NN^n$, 
the monomial $x^\alpha := x_1^{\alpha_1}\cdots x_n^{\alpha_n}$ with exponent $\alpha$ is
the map $x^{\alpha}:\CC^n \to \CC$. 
A sparse polynomial is a linear combination of monomials. 
Let $\AS_\bullet =(\AS_1,\dots, \AS_k)$ denote a $k$-tuple of nonempty finite subsets of $\NN^n$.
A general sparse polynomial system of equations with support $\AS_\bullet$ is given by 
\begin{align*}
\begin{split}
    \sum_{\alpha \in \AS_1} c_{1,\alpha}x^\alpha = \ldots = \sum_{\alpha \in \AS_k} c_{k,\alpha}x^\alpha =0, 
\end{split}
\end{align*}
where the coefficients $\{c_{i,\alpha } \}_{\alpha\in \AS_i, i\in [k]}$ are general. 

\begin{rem}
The concept of genericity is fundamental in applied algebraic geometry.
Throughout the rest of this paper we consider a \demph{general sparse polynomial system} $F$ and \JIR{\demph{general data vector}} $u$. Formally, we require that the coefficients of the polynomials of $F$ and entries of $u$ lie in a dense Zariski open set.
\end{rem}

Related work has considered a similar optimization problem 
\begin{equation}\label{eq:optimization-problem}
 \min_{x \in \mathbb{R}^n} \ g(x) \quad \text{subject to} \quad x \in \mathcal{X} 
\end{equation}
where $\mathcal{X} $ is a real algebraic variety and $g$ is a specified objective function. A particular choice of $g$ that is of interest is when $g = \lVert x - u \rVert_2^2$ for a point $u \in \mathbb{R}^n$. This is called the \demph{Euclidean distance function} and the number of critical points to this optimization problem for general $u$ is called the \demph{ED degree of $\mathcal{X}$}. The study of ED degrees began with \cite{DraismaTheEDD} and initial bounds on the ED degree of a variety were given in \cite{MR3451425}. Other work has found the ED degree for real algebraic groups \cite{baaijensRealAlgGroups}, Fermat hypersurfaces \cite{leeFermat}, orthogonally invariant matrices \cite{drusvyatskiyOrthogonally}, smooth complex projective varieties \cite{aluffiEDComplex}, the multiview variety \cite{maximMultiview}, and when $\mathcal{X}$ is a hypersurface \cite{breiding2020euclidean}. Further work has considered instances of this problem when the data $u$ are not general \cite{maximDefect} as well as when the semidefinite relaxation is tight~\cite{MR4115654}.

A final connection is when the objective function in \eqref{eq:optimization-problem} is a polynomial. %
In this case, the number of critical points is called the \demph{algebraic degree of the optimization problem. }
In \cite{MR2507133},  
the algebraic degree of \eqref{eq:optimization-problem} 
is considered when $\mathcal{X}=V(f_1,\dots,f_k)$ with $f_i$ and $g$ are all generic polynomials of some degree. By ~\cite[Proposition 2.1]{MR2507133}
the number of solutions $(x^*,\lambda_1^*,\dots,\lambda_k^*)\in \CC^{n+k}$ to the the \demph{Karush-Kuhn-Tucker (KKT)-system}
\begin{align*}
    \nabla g(x^*)+\sum_{i=1}^k \lambda_i^* \nabla f_i(x^*) = 0 \\ 
f_1(x^*)=\cdots = f_k(x^*)=0
\end{align*}
is the algebraic degree. Moreover, a formula for this degree is given in \cite[Theorem 2.2]{MR2507133} in terms of the degrees of $g$ and $f_1,\dots,f_k$.
Other formulas for many classes of convex polynomial optimization problems are given in 
\cite{MR2496496}
and \cite{MR2546336}.
Related topics and background on algebraic optimization problems and the corresponding convex geometry can be found in \cite{blekhermanConvexAGBook}.

\section{The ML degree of sparse systems}\label{s:mldegree}

Let $F:\mathbb{R}^n \to \mathbb{R}^k$ be a sparse polynomial system with general coefficients. Let $u \in \mathbb{R}_{> 0}^n$ be a general point. Here $u$ is the data and $F = \langle f_1,\ldots, f_k \rangle$ gives the model.
We want to solve the maximum likelihood optimization problem:
\begin{equation}\label{mleoptimization}
\sup_{x\in \mathbb{R}_{> 0}^n} \  \sum_{i=1}^n u_i \log(x_i) \qquad \text{ subject to } \qquad x \in \mathcal{V}(F). \tag{MLE}
\end{equation}

One approach to solving \eqref{mleoptimization} is to find all critical points which can be done using Lagrange multipliers. 
The \demph{Lagrangian function} for  \eqref{mleoptimization} is defined as 
\begin{align}
    \lagObjective(x_1,\ldots, x_n, \lambda_1,\ldots, \lambda_k) &:= \sum_{i=1}^n u_i \log(x_i) - \sum_{j=1}^k \lambda_j f_j.
\end{align}
To find all complex critical points of \eqref{mleoptimization} we solve the square polynomial system 
$\mathcal{L}: \mathbb{C}^{n + k} \to \mathbb{C}^{n+k}$ 
obtained by taking the partial derivatives of $\lagObjective$. The partial derivatives~are 

\begin{align}
\frac{\partial}{\partial x_i} \lagObjective &= \frac{u_i}{x_i} - \frac{\partial}{\partial x_i} \Big( \sum_{j=1}^k \lambda_j f_j \Big), &&\quad i \in [n], \\
\frac{\partial }{\partial \lambda_j} \lagObjective &= -f_j, &&\quad j \in [k]. \label{eq:partiallambda}
\end{align}

Multiplying $\frac{\partial}{\partial x_i} \lagObjective$ by $x_i$
clears the denominators to get the polynomials

\begin{align}\label{lagrangemulti}
    \ell_i &:= x_i \cdot \frac{\partial }{\partial x_i} \lagObjective
    = u_i - x_i \sum_{j=1}^k \lambda_j \frac{\partial}{\partial x_i} (f_j ), \quad i\in[n]. 
\end{align}

Using the notation in \eqref{eq:partiallambda}-\eqref{lagrangemulti},
the \demph{ML system} of $F$ is 
\begin{equation}\label{eq:MLsystem}
    \mathcal{L}(F) = \langle \ell_1,\ldots, \ell_n, f_1,\ldots, f_k \rangle,
\end{equation}
a system in $n+k$ unknowns and $n+k$ equations.
In the literature, the ML system is also known as the Lagrange likelihood equations. We use the former terminology for brevity.
The \demph{ML degree} of $F$ is the number of complex (real or non-real) solutions to $\mathcal{L}(F)$ for generic data.

The following proposition shows that the ML degree of a sparse polynomial system is well defined. 

\begin{prop}\label{prop:finite_solutions} 
For a general sparse polynomial system $F = \langle f_1,\ldots, f_k\rangle$ 
and for generic data $u$, 
the corresponding ML system has finitely many solutions in $\CC^n\times \CC^k$. 
Moreover, all solutions to the ML system are in 
$ ( \CC^* )^n \times (\CC^*)^k$.

\end{prop} 
\begin{proof}
This proof uses genericity in two different ways. 
First we use genericity of the coefficients of $f_1,\dots,f_k$. By Bertini's Theorem~\cite[Ch. III,\S 10.9.2]{MR0463157},
the variety of $\langle f_1,\dots, f_k \rangle$ saturated by the coordinate hyperplanes is either empty or codimension $k$. 
Denote this variety by $X$. 
Moreover, by Bertini's Theorem, if $k<n$, then the variety $X$ is irreducible.   

The polynomials $\ell_1,\dots,\ell_n$ give a map $X\times \mathbb{C}^k_\lambda\to \mathbb{C}^n_u$. 
The source of this map is $n$-dimensional and irreducible and therefore the image is at most $n$-dimensional and irreducible. 

Now we use genericity of the data. 
If the image is $n$ dimensional, then a fiber over a generic point in $\CC^n$ is zero dimensional.
This means the ML system for generic data has finitely many solutions. 
On the other hand, if the image is lower dimensional, then the fiber over a generic point is empty. In such a case the ML degree is zero.

Since $X$ is defined by saturating by the coordinate hyperplanes, we must show that there are still only finitely many solutions to the ML System in $(\CC^n\setminus (\CC^*)^n)\times \CC^k$.  
By the data being generic, we may assume the $u_i$ coordinate is nonzero. 
For $i=1,\dots, n$, having $u_i\neq 0$  and 
$\ell_i=0$ 
implies that the $x_i$ coordinate of the solution is not zero.
Therefore all solutions to the ML system are in $(\CC^*)^n\times \CC^k$.

We have shown the first statement and part of the second statement. It remains to show that there are no solutions with $\lambda_i = 0$ for $i \in [k]$. 
If we assume $\lambda_k^*=0$ by way of contradiction, 
then $(x_1^*,\dots,x_n^*,\lambda_1^*,\dots,\lambda_{k-1}^*)$
is a solution to the ML system of $f_1,\dots,f_{k-1}$.
By the argument above, 
this new ML system has finitely many solutions. By the genericity of $f_k$, none of these solutions will satisfy $f_k(x^*) = 0$.
\end{proof}

We remark that the arguments used in the first half of the proof are analogous to the ones presented in \cite[Proposition 3]{HostenSolving}.

\subsection{
Newton polytopes of likelihood equations
and the algebraic torus
}\label{ss:NPofLE}

We want to use existing results on sparse polynomial systems from algebraic geometry. To do this, we first need a few definitions.

\begin{defn}
For a polynomial $f\in \CC[x_1,\dots,x_n]$,
we define the \demph{Newton polytope} of $f$ as the convex hull of the set of exponent vectors of $f$. We denote this by $\newt(f),$ and we let $\Vertex(f)$ denote the set of vertices of $\newt(f)$.

\end{defn}

The next lemma describes the Newton polytopes of the ML system~\eqref{eq:MLsystem}. 
We use the notation $x_1\cdots x_n \mid f$ when there exists a polynomial $g$ such that $x_1\cdots x_n\cdot g=f$.

\begin{lem}\label{lem:newt_likelihood}
Consider a sparse polynomial system $F = \langle f_1,\ldots, f_k \rangle$. If $x_1 \cdots x_n \mid f_j$ for all $j \in [k]$, then for every $i \in [n]$ and $j \in [k]$, $\newt(f_j)= \newt(x_i \frac{\partial}{\partial x_i} f_j)$. Moreover, for every $i \in [n]$ the Newton polytope of $\ell_i$ is equal to
\[
 \newt(\ell_i) = \conv ( \{ 0_{n+k}  \} \cup \Vertex(\lambda_1 f_1)\cup\dots \cup \Vertex(\lambda_k f_k)   ).
\]
\end{lem}
\begin{proof}

The proof of the first statement follows from the fact that a Newton polytope is determined by its vertices, and that 
\[
\frac{\partial}{\partial x_i} 
\left(
x_1^{\alpha_1}\cdots x_n^{\alpha_n} 
\right)
=
\begin{cases}
0&  \alpha_i=0, \\ 
\alpha_i \frac{x_1^{\alpha_1}\cdots x_n^{\alpha_n}}{x_i}   & \text{otherwise}.
\end{cases}
\]
The proof of the second statement follows from the the definition of the likelihood equations and Newton polytopes.

\end{proof}

The following example provides an intuitive description of
Lemma~\ref{lem:newt_likelihood}.

\begin{ex}\label{ex:newtFvsFhat}

Let $F = f =  \langle  2x^4 + 3y^3 - 5\rangle$, and consider variable ordering $(x, y, \lambda)$.~Then 
the Newton polytopes given by $\mathcal{L}(f)$ are
\begin{align*}
\newt(f) &= \conv(\{(4, 0, 0), (0, 3, 0), (0, 0, 0)\}), \\
\newt(\ell_1) &= \conv(\{(0, 0, 0), (4, 0, 1)\}), \text{ and } \\
\newt(\ell_2) &=\conv( \{(0, 0, 0), (0, 3, 1)\}). 
\end{align*}
These are different from the Newton polytopes coming from $\mathcal{L}(\hat f)$, where 
    $\hat{f} =  \langle  xy \cdot f\rangle$:
\begin{align*}
    \newt(\hat{f})  &= \conv(\{(5, 1, 0), (1, 4, 0), (1, 1, 0)\}), \text{ and } \\
    \newt(\hat{\ell_1}) = \newt(\hat{\ell_2}) &= \conv(\{(5, 1, 1), (1, 4, 1), (1, 1, 1), (0, 0, 0)\}). 
\end{align*}

\begin{figure}[htbp]
\centering
\includegraphics[width=5.2cm]{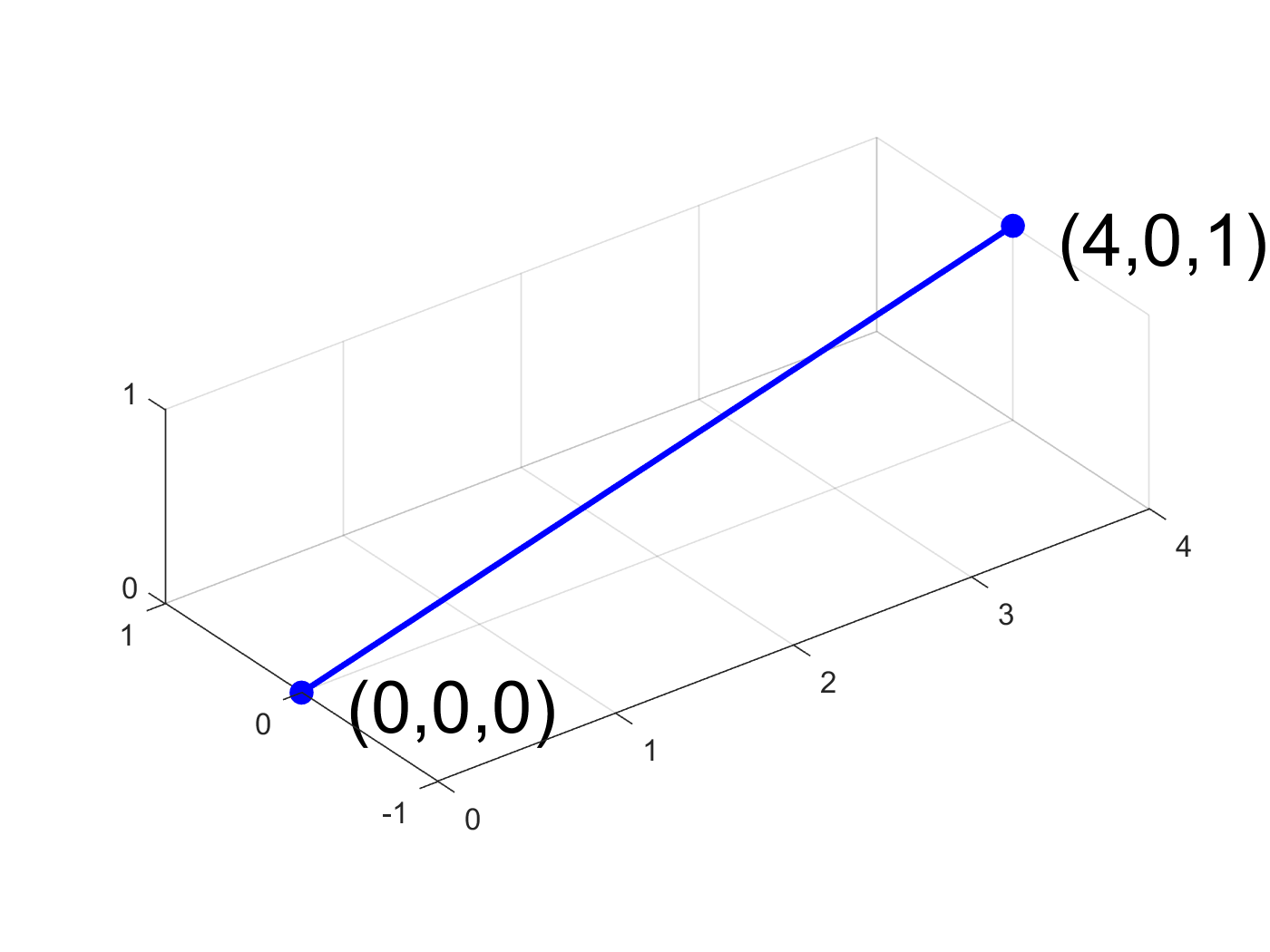}
\includegraphics[width=5.2cm]{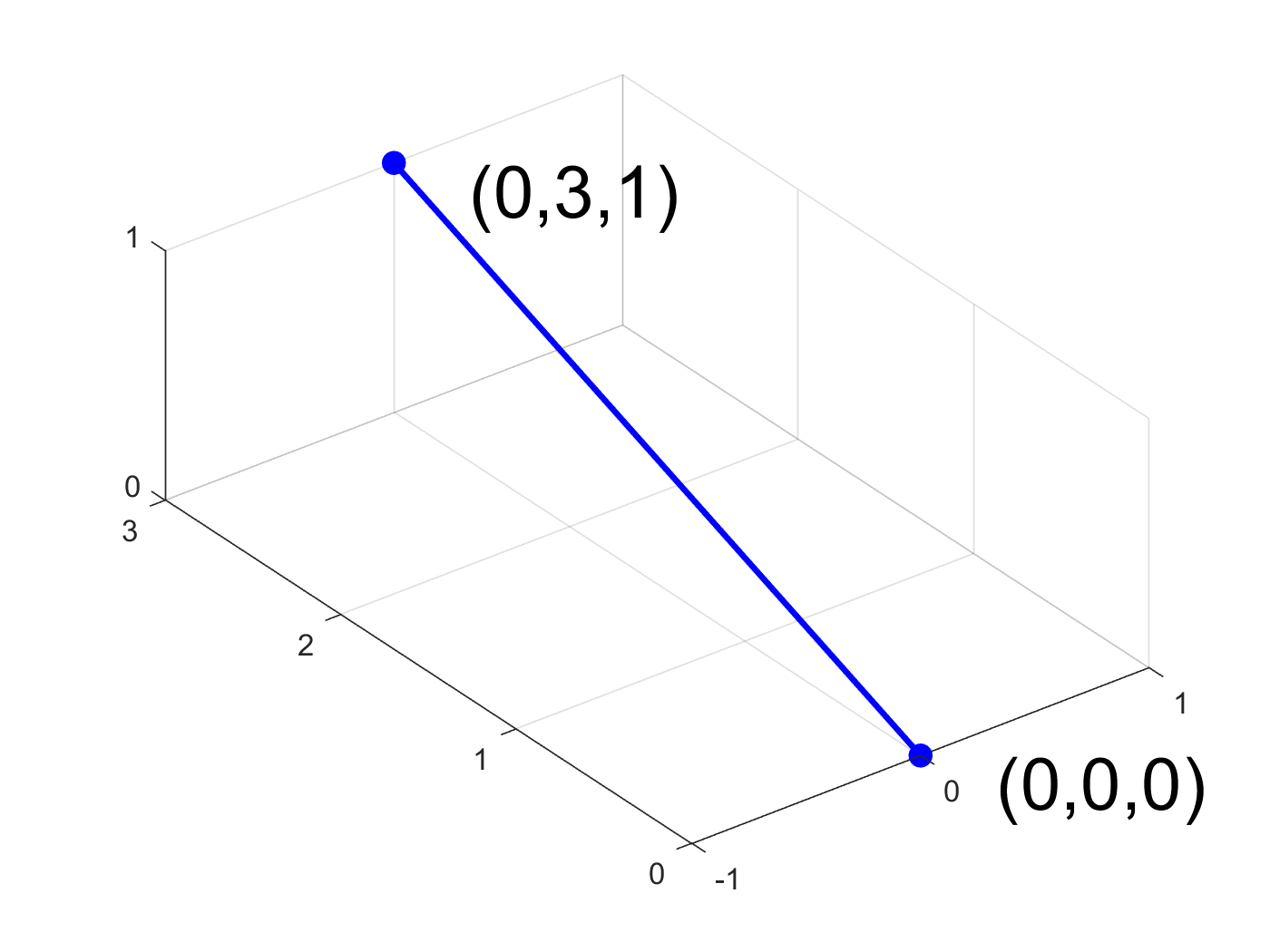}
\includegraphics[width =5.2cm]{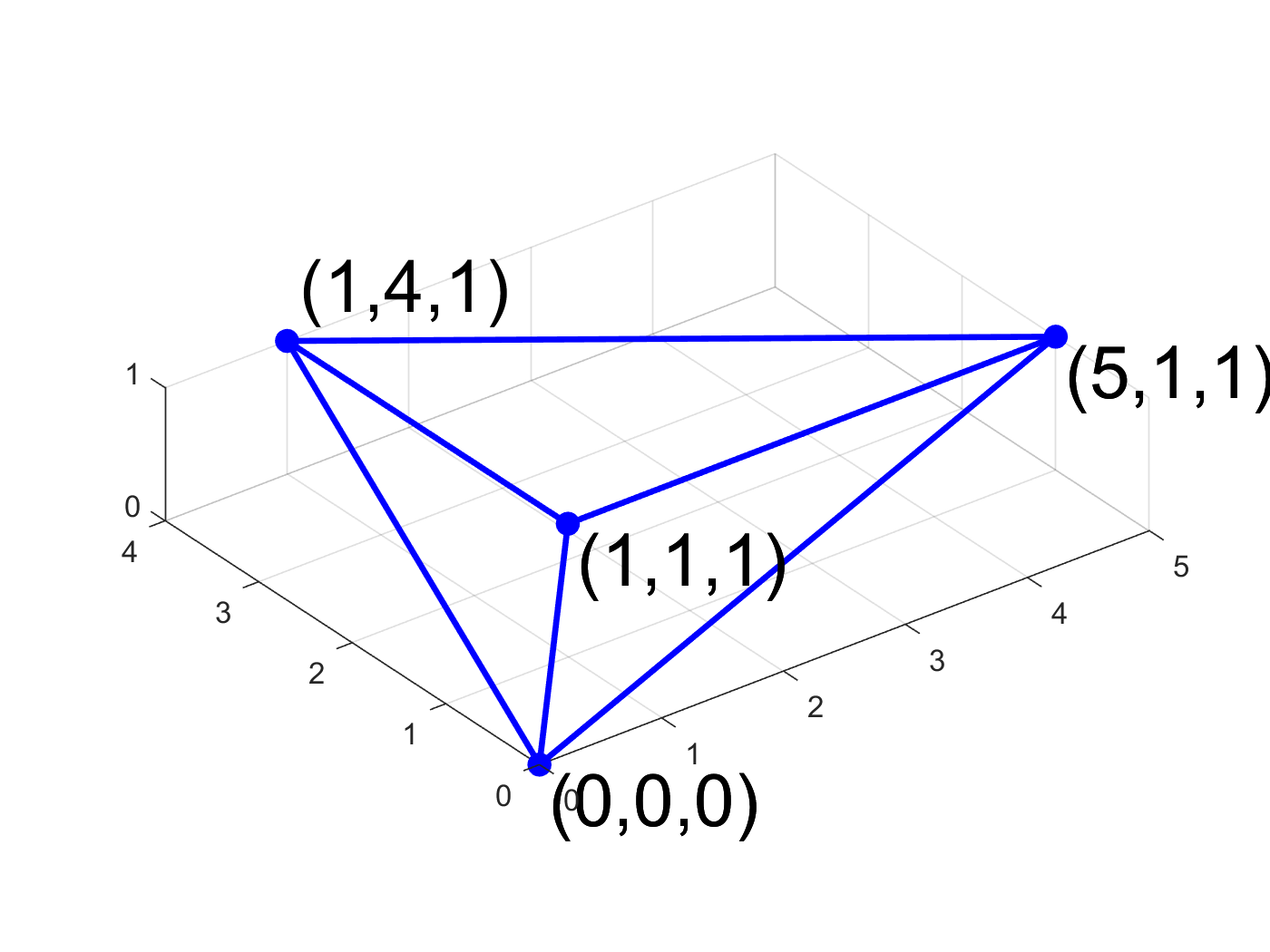}
\caption{$\newt(\ell_1), \newt(\ell_2)$ and $\newt(\hat{\ell_1}) = \newt(\hat{\ell_2})$ from Example~\ref{ex:newtFvsFhat}}
\label{fig:newtpoly}
\end{figure}

\end{ex}

The following proposition shows that the  assumption $x_1 \cdots x_n \mid f_j$ for all $j \in [k]$ in Lemma~\ref{lem:newt_likelihood} is not an issue. 

\begin{prop}\label{prop:Fhat_equals_F}
Let $F = \langle f_1,\ldots,f_k \rangle$ and $\hat{F} = \langle\hat{f}_1, \ldots, \hat{f}_k\rangle$ where 
$f_j \in \mathbb{C}[x_1,\ldots, x_n]$ and $\hat{f}_j = x_1\cdots x_n \cdot f_j$ for $j \in [k]$. The ML degree of $F$ equals the ML degree of $\hat{F}$. 
\end{prop}

\begin{proof}
Recall the definition of ML system in \eqref{eq:MLsystem}, and
 let 
\[
\mathcal{L}(F) = \langle \ell_1,\ldots,\ell_n,f_1,\ldots,f_k \rangle \text{ and }\mathcal{L}(\hat{F})= \langle \hat{\ell}_1,\ldots,\hat{\ell}_n,\hat{f}_1,\ldots,\hat{f}_k \rangle.
\] 
By Proposition~\ref{prop:finite_solutions} it suffices to show that there is a bijection between $\mathcal{V}(\mathcal{L}(F)) \cap (\mathbb{C}^*)^{n+k} $ and $\mathcal{V}(\mathcal{L}(\hat{F})) \cap (\mathbb{C}^*)^{n+k}$. We claim such a bijection is given by
\begin{align*}
    \phi: \mathcal{V}(\mathcal{L}(F))\cap(\mathbb{C}^*)^{n+k} &\to \mathcal{V}(\mathcal{L}(\hat{F}))\cap (\mathbb{C}^*)^{n+k} \\
    (x_1,\ldots,x_n,\lambda_1,\ldots,\lambda_k) &\mapsto (x_1,\ldots, x_n,\frac{\lambda_1}{x_1\cdots x_n}, \ldots, \frac{\lambda_k}{x_1\cdots x_n}) 
\end{align*}
We need to show that $\phi$ is well defined. Since we assume 
$(x,\lambda) \in  (\mathbb{C}^*)^{n+k}$,
$\frac{\lambda_i}{x_1\cdots x_n}$ is well defined. Now observe that if $f_j(x_1,\ldots,x_n) = 0$ then $\hat{f}_j = x_1\cdots x_n \cdot f_j(x_1,\ldots,x_n) = 0$ so we only need to show $\hat{\ell}_i$ vanishes on the image of $\phi$. By definition,
\begin{align*}
    \hat{\ell}_i &= u_i - x_i \cdot \sum_{j=1}^k \lambda_j \frac{\partial}{\partial x_i}(x_1\cdots x_n f_j) \\
    &= u_i - x_i \cdot \sum_{j=1}^k \lambda_j (x_1\cdots x_{i-1}x_{i+1}\cdots x_n f_j + x_1\cdots x_n \frac{\partial}{\partial x_i}(f_j)).
\end{align*}

Since $f_j(x_1,\ldots,x_n) = 0$ the first term in the summand vanishes. Substituting $\lambda_j \mapsto \frac{\lambda_j}{x_1\cdots x_n}$, the result is then clear.

Consider
\begin{align*}
    \phi^{-1}: \mathcal{V}(\mathcal{L}(F))\cap(\mathbb{C}^*)^{n+k} &\to \mathcal{V}(\mathcal{L}(\hat{F}))\cap (\mathbb{C}^*)^{n+k} \\
    (x_1,\ldots, x_n, \lambda_1,\ldots, \lambda_k) &\mapsto (x_1,\ldots, x_n,x_1 \cdots x_n\lambda_1,\ldots, x_1 \cdots x_n \lambda_k).
\end{align*}

It is clear that the map $\phi \circ \phi^{-1} = \phi^{-1} \circ \phi$ is the identity, and that
\[
\phi^{-1}(x_1,\ldots,x_n,\lambda_1,\ldots,\lambda_k) \in \mathcal{V}(\mathcal{L}(F)) \cap (\mathbb{C}^*)^{n + k}.
\]
\end{proof}

\subsection{Initial systems of the likelihood equations}\label{ss:init-likelihood-equations}
We now consider the geometry of the Newton polytopes of the likelihood equations and how it relates to the number of $\mathbb{C}^*$ solutions to these equations. 

Given a nonzero vector $ w \in \mathbb{Z}^n$ and a polytope $P \subseteq \mathbb{R}^n$, we denote $P_w$ as the \demph{face exposed} by $w$ and $\val_w(P)$ the \demph{value} $w$ takes on this face. Specifically:
\[P_w = \{x \in P : \langle w,x \rangle \leq \langle w,y \rangle \ \text{ for all } \ y \in P \} \quad \text{and } \quad \val_w(P) = \min_{x \in P} \langle w,x \rangle,\]
with $\langle (w_1,\dots,w_n),(x_1,\dots,x_n) \rangle
:=w_1x_1+\cdots w_nx_n$.
If $f = \sum_{\alpha \in \newt(f)} c_\alpha x^{\alpha}$, we call
\[
\init_w(f) = \sum_{\alpha \in (\newt(f))_w} c_\alpha x^{\alpha}
\]
the \demph{initial polynomial} of $f$. For convenience, let $\newt_w(f)$ denote $(\newt(f))_w$ and $\val_w(f) = \val_w(\newt(f))$.
For more background on initial polynomials, see \cite[Chapter~2]{idealscox}.

For convex bodies 
$K_1,\dots,K_n$
in $\mathbb{R}^n$, 
consider the Minkowski sum 
$\mu_1 K_1 +\cdots +\mu_n K_n.$
The volume of
$\mu_1 K_1 +\cdots +\mu_n K_n$
as a function of $\mu_1,\dots, \mu_n$
is a homogeneous polynomial $Q(\mu_1,\dots,\mu_n)$ of degree $n$.
The \demph{mixed volume} of 
$K_1,\dots,K_n$, denoted $\mvol(K_1,\ldots, K_n)$,
is defined to be 
the coefficient of $\mu_1\cdots\mu_n$ in $Q$.
For more details about mixed volumes see~\cite{mvoltext}. 

There are three important properties of the mixed volume which we wish to highlight: 
\begin{enumerate}
    \item \textbf{Translation Invariance:} $\mvol(K_1,\ldots, K_n) = \mvol(a + K_1,\ldots, K_n)$ for $a \in \mathbb{R}^n$, 
    \item \textbf{Monotonicity:} $\mvol(\tilde K_1,K_2\ldots,K_n) \leq \mvol( {K}_1,\ldots, K_n)$ when $\tilde K_1 \subseteq {K}_1$,
    \item \textbf{Special Linear Invariance:} $\mvol(K_1,\ldots, K_n) = \mvol(\phi K_1,\ldots, \phi K_n)$ for any 
   $\phi$ in the special linear group $\SL_n(\mathbb{R}$).
\end{enumerate}

Recent work has analyzed when the monotonicity inequality is strict \cite{bihan2019criteria}.
The connection between the convex geometry of a polynomial system and the number of $\mathbb{C}^*$ solutions to this system was made in a sequence of papers \cite{bernshtein1979the,khovanskii1978newton,kouchnirenko1976polyedres}.

\begin{theorem}[BKK bound 
]\label{thm:BKK}
Let $G = \langle g_1,\dots,g_n\rangle$ be a general sparse polynomial system in
$\mathbb{C}[x_1,\dots,x_n]$ and let $P_1,\dots,P_n$ be their respective Newton polytopes. The number of $\CC^*$-solutions to $F=0$ is equal to $\mvol(P_1,\ldots,P_n)$. Moreover, $\mvol(P_1,\ldots,P_n)$ is an  upper bound for the number of isolated solutions in $(\CC^*)^n$ for a system with arbitrary coefficients. {If for every nonzero $w \in \mathbb{Z}^n$,} the initial systems $\init_w(g_1,...,g_n)$ have no solutions in $(\CC^*)^n$, then all the roots of the system are isolated. 
\end{theorem}

By Proposition~\ref{prop:finite_solutions} we know that for a general sparse polynomial system $F$ and data vector $u$, there are finitely many complex solutions to the likelihood equations and that all such complex solutions live in the torus. Therefore, we would like to use Theorem~\ref{thm:BKK} to identify the ML degree of $F$. To do this we need some preliminary results.

By Lemma \ref{lem:newt_likelihood}, 
if $x_1 \cdots x_n \mid f_j$ for all $j \in [k]$
then
\[
\newt(\ell_j) = \newt(\ell_i)
\] for $i,j \in [n]$. Call this polytope $P$. Given some 
nonzero weight vector $w\in\mathbb{Z}^{n+k}$,
we would like to determine which face of $P$ is exposed by $w$, based on which faces of $\newt(f_1),\ldots, \newt(f_k)$ are exposed by $w$.

\begin{lem}\label{lem:min_faces}

Let $F=\langle f_1,\dots,f_k\rangle $ denote a general sparse polynomial system.
Let $\tilde{e}_j \in \mathbb{R}^{n+k}$ be the vector with $(n+j)$-th entry equal to $1$ and all other entries equal to ~$0$. 
Suppose $w$ is a nonzero weight vector in $\mathbb{Z}^{n+k}$.

If  
$x_1 \cdots x_n \mid f_j$ for all $j \in [k]$,
then
up to reordering the $f_1,\ldots, f_k$, $w$ exposes $P$ on one of the following faces:

\begin{enumerate}
    \item\label{case:one} the origin, 
    \item\label{case:two} $ \conv ( \tilde{e}_1 + \newt_w(f_1), \ldots , \tilde{e}_t + \newt_w(f_t))$  for some $t \in [k]$,
    \item\label{case:three} $ \conv( 0, \tilde{e}_1 + \newt_w(f_1), \ldots , \tilde{e}_t + \newt_w(f_t) )$  for some $t \in [k]$.
\end{enumerate}
\end{lem}

\begin{proof}
Fix a nonzero weight vector 
$w=(a,b) \in \mathbb{R}^n\times\mathbb{R}^k$ and suppose $(v,0_k) \in \newt(f_1) \setminus\newt_w(f_1)$. 
From the description of $P$ in Lemma~\ref{lem:newt_likelihood} we have that
\[\val_w(P) \in \{0, b_1 + \val_w(f_1),\ldots, b_k + \val_w(f_k) \}. \] 
If $b_j + \val_w(f_j) >0$ for all $j \in [k]$, then $P$ is exposed at the origin, so we are in Case~\ref{case:one}. If $b_1 + \val_w(f_1) = \cdots = b_t + \val_w(f_t)  = \gamma <0$ for some $t \in [k]$, where $b_j + \val_w(f_j) > \gamma$ for all $t+1 \leq j \leq k$, then we are in Case~\ref{case:two}. If above $\gamma = 0$, then we are in Case~\ref{case:three}.
\end{proof}

We illustrate Lemma~\ref{lem:min_faces} with the following example. 
\begin{ex}\label{ex:min_faces}
Recall the ML system $\mathcal{L}(\hat f)$  from Example~\ref{ex:newtFvsFhat} where $\hat f =  \langle  xy(2x^4 + 3y^3 - 5)\rangle $ and  $P$ from  Figure~\ref{fig:newtpoly}. 
Consider the three weight vectors 
\[
w_1 = (-3, 14, 3), \quad
w_2 = (-3, -4, 3), \quad 
w_3 = (-3, 12, 3).
\]
The respective exposed faces of $P$ for these weight vectors are
\[
P_{w_1} = \{(0,0,0)\},\quad 
P_{w_2} = \conv(\{(5, 1, 1), (1, 4, 1)\}), \quad
P_{w_3} = \conv(\{(0,0,0),(5, 1, 1) \}),
\]
and are shown in red in Figure~\ref{fig:min_faces}.
Each $P_{w_i}$ corresponds to one of the three cases in Lemma~\ref{lem:min_faces}. Namely, $P_{w_1}$ is the origin; 
 $P_{w_2}$ is in Case~\ref{case:two}; and 
 $P_{w_3}$ is in Case~\ref{case:three}.

\begin{figure}[htbp]
\centering
\includegraphics[width=5.2cm]{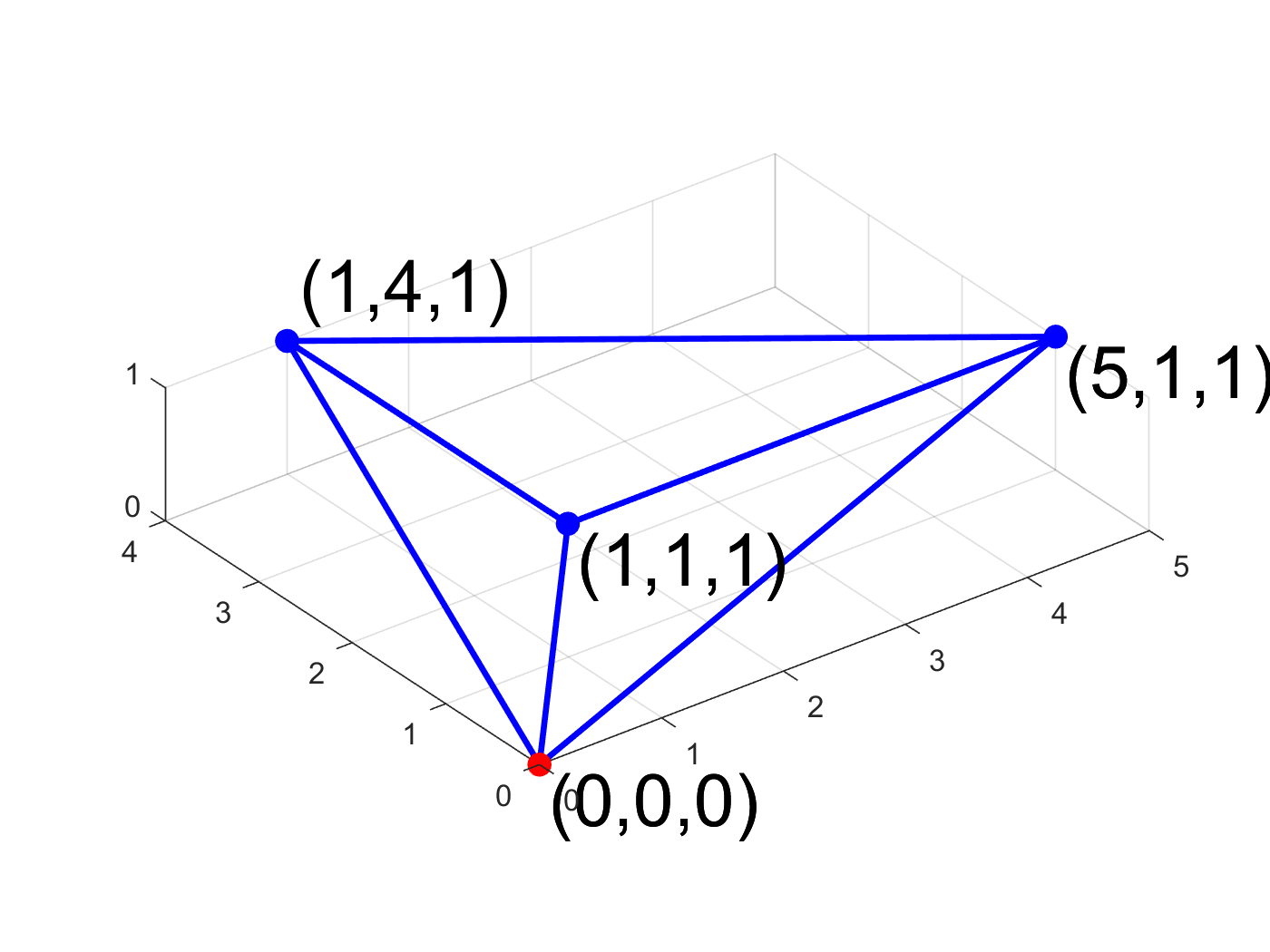}
\includegraphics[width=5.2cm]{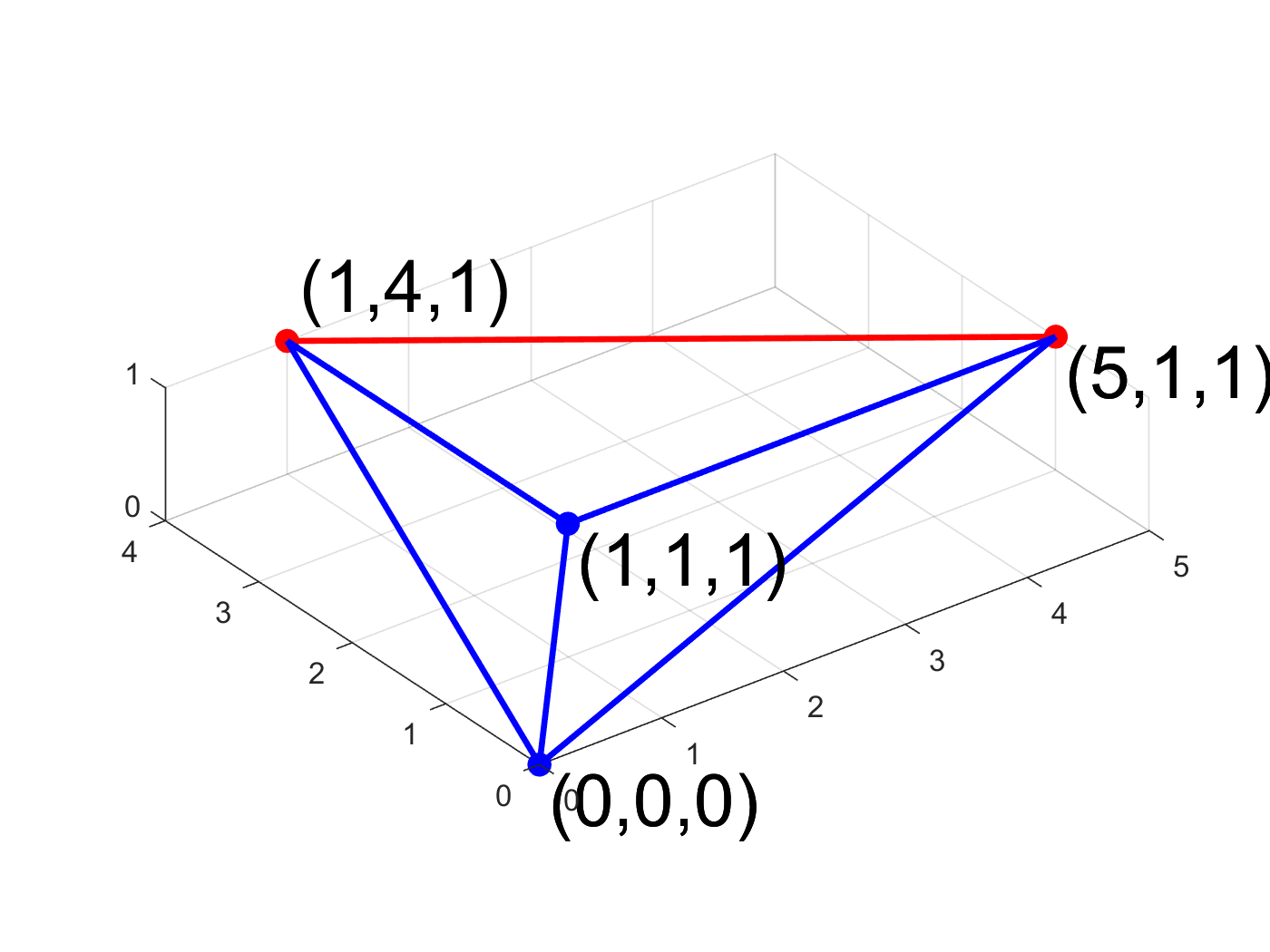}
\includegraphics[width=5.2cm]{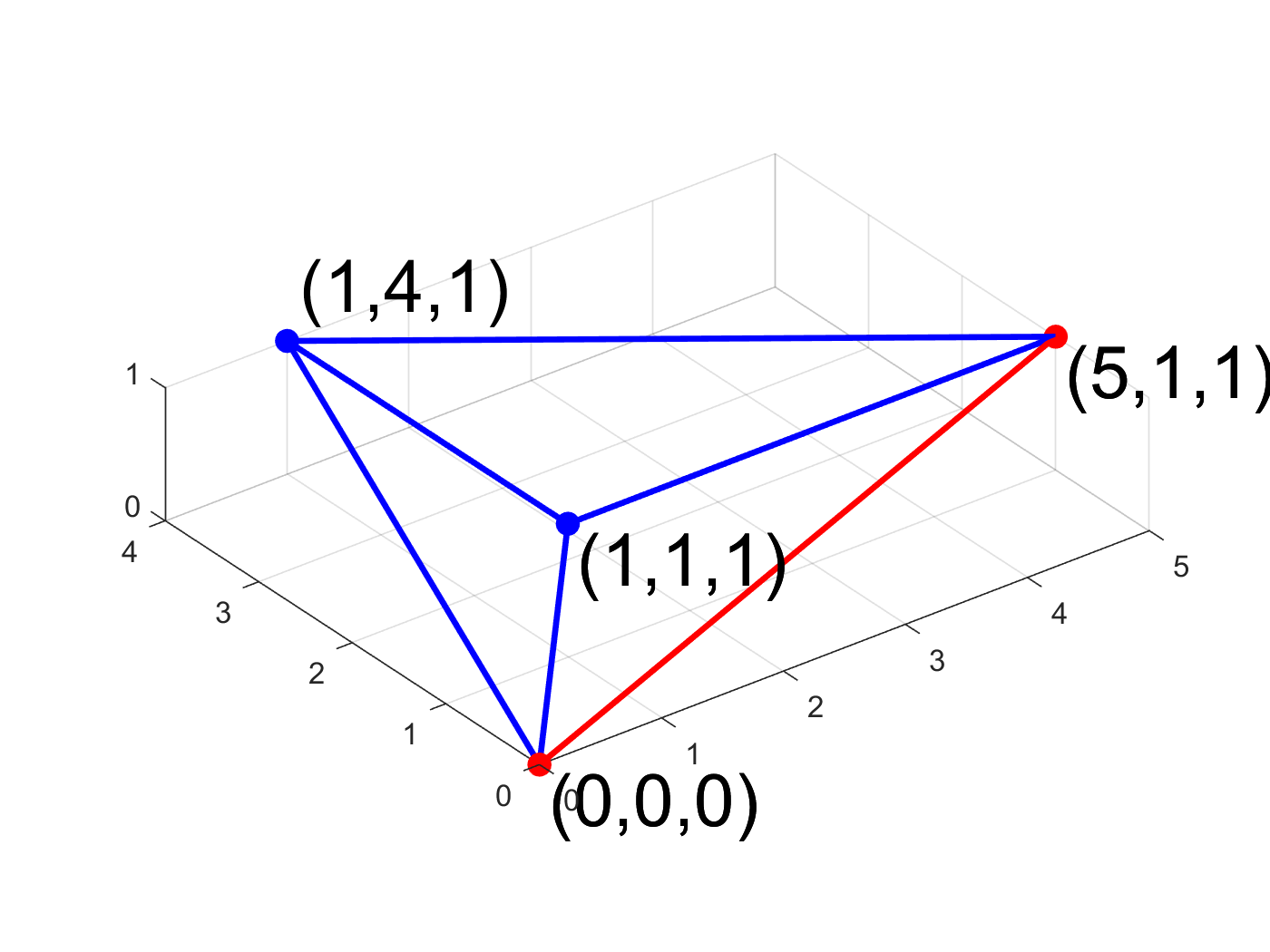}
\caption{$P_{w_1}, P_{w_2}$ and $P_{w_3}$ from Example~\ref{ex:min_faces}.}\label{fig:min_faces}
\end{figure}

\end{ex}

We now need to show that for each of the three cases outlined in Lemma \ref{lem:min_faces}, there are no $\mathbb{C}^*$ solutions to the corresponding initial system.

\begin{lem}\label{lem:case1}
Let $F=\langle f_1,\dots,f_k\rangle $ denote a general sparse polynomial system.
If  
$x_1 \cdots x_n \mid f_j$ for all $j \in [k]$,
then
there are no $\mathbb{C}^*$ solutions to $\init_w(\mathcal{L}(F))$ when $P_w$ is as in Case~\ref{case:one}.
\end{lem}
\begin{proof}

Recall that in Case~\ref{case:one} in 
Lemma~\ref{lem:min_faces}, $P_w$ is the origin.
In this case we have $\init_w(\ell_i) = u_i = 0$. Since generally $u_i \neq 0$, this initial system has no solutions. 
\end{proof}

\begin{lem}\label{lem:case2}
Let $F=\langle f_1,\dots,f_k\rangle$ denote a general sparse polynomial system.
If  
$x_1 \cdots x_n \mid f_j$ for all $j \in [k]$,
then
there are no $\mathbb{C}^*$ solutions to 
$\init_w(\mathcal{L}(F))$ when $P_w$ is as in Case~\ref{case:two}.
\end{lem}
\begin{proof}
Recall that in Case~\ref{case:two} in 
Lemma~\ref{lem:min_faces}, $P_w$ is 
$ \conv ( \tilde{e}_1 + \newt_w(f_1), \ldots , \tilde{e}_t + \newt_w(f_t))$  for some $t\in [k]$.

Let $\tilde{f}_j = \init_w(f_j)$ for $j \in [k]$. We consider the following as a subsystem of $\init_w( \mathcal{L}(F))$:
\begin{align*}
    \tilde{f_1} &= \ldots = \tilde{f_t} = 0, \\
    x_1 \big( \sum_{j=1}^t \lambda_j \frac{\partial}{\partial x_1} (\tilde{f_j})  \Big)&= \ldots = x_n \big( \sum_{j=1}^t \lambda_j \frac{\partial}{\partial x_n} (\tilde{f}_j) \Big) = 0.
\end{align*}

Since we only consider $\mathbb{C}^*$ solutions, this reduces to 
\begin{align*}
    \tilde{f_1} &= \ldots = \tilde{f_t} = 0, \\
    \sum_{j=1}^t \lambda_j \frac{\partial}{\partial x_1} (\tilde{f_j}) &= \ldots =  \sum_{j=1}^t \lambda_j \frac{\partial}{\partial x_n} (\tilde{f_j})  = 0.
\end{align*}

By Bertini's Theorem~\cite[Ch. III,\S 10.9.2]{MR0463157},
 the variety cut out by $\tilde{f}_1 = 0, \ldots , \tilde{f}_t = 0$ has codimension $t$ in $(\mathbb{C}^*)^n$ and $\mathcal{V}(\tilde{f}_1,\ldots, \tilde{f}_t)$ has no singular solutions in the torus. 
So this initial system has no $\mathbb{C}^*$ solutions. 

\end{proof}

Before we consider the final case of Lemma~\ref{lem:min_faces}, we need a preliminary lemma.

\begin{lem}\label{lem:w_tilde_0}
Let $F=\langle f_1,\dots,f_k\rangle$ denote a general sparse polynomial system where  
$x_1 \cdots x_n \mid f_j$ for all $j \in [k]$.
Furthermore, let $w = (a,b) \in \mathbb{Z}^n \times \mathbb{Z}^k$ be a nonzero weight vector. If $a = 0 $ then there are no $\mathbb{C}^*$ solutions to
$\init_w(\mathcal{L}(F))$.

\end{lem}
\begin{proof}
Under the assumption $a = 0$, 
\[
\newt_w(f_j) = \newt(f_j) \text{ and } \val_w(f_j) = 0\text{ for all } j \in [k].
\]
Recall from the proof of Lemma~\ref{lem:min_faces},
\[ \val_w(P) \in \{0, b_1 + \val_w(f_1), \ldots, b_k + \val_w(f_k) \}. \]
Since $\val_w(f_j) = 0$, this gives that $\val_w(P) \in \{0, b_1,\ldots, b_k \}$. If any $b_j<0$ or all $b_j>0$ for $j \in [k]$, then by Lemmas~\ref{lem:case2} and \ref{lem:case1} there are no $\mathbb{C}^*$ solutions to the initial system. 

It remains to consider when 
\[b_1 = \ldots = b_t = 0 \ \text{ and } \ b_{t+1},\ldots, b_k>0. \] 
Note that $t <k$, because otherwise $b$ would be the all zeros vector, which is not allowed. Observe that $\init_w( \ell_1,\ldots, \ell_n, f_1,\ldots, f_t)$ is equal to the ML system of $(f_1,\ldots, f_t)$. By Proposition~\ref{prop:finite_solutions} there are finitely many solutions to this ML system. Since $f_{t+1},\ldots, f_{k}$ are general, 
their respective hypersurfaces do not intersect the variety of this Lagrange system.
\end{proof}

\begin{lem}\label{lem:case3}
Let $F=\langle f_1,\dots,f_k\rangle$ denote a general sparse polynomial system.
If  
$x_1 \cdots x_n \mid f_j$ for all $j \in [k]$,
then
there are no $\mathbb{C}^*$ solutions to $\init_w(\mathcal{L}(F))$ when $P_w$ is as in Case~\ref{case:three}.
\end{lem}

\begin{proof}
Recall from  Case~\ref{case:three} in 
Lemma~\ref{lem:min_faces}, $P_w$ is 
$ \conv( 0, \tilde{e}_1 + \newt_w(f_1), \ldots , \tilde{e}_t + \newt_w(f_t) )$  for some $t\in [k]$.
Let $\tilde{f}_j = \init_w(f_j)$ for $j \in [k]$.

We consider the subsystem of $\init_w(\mathcal{L}(F))$ given by: 
\begin{align*}
    \tilde{f_1} = \ldots = \tilde{f_t} &= 0 \\
    x_1\sum_{j=1}^t \lambda_j \frac{\partial}{\partial x_1} (\tilde{f_j}) &= u_1 \\
    \vdots \qquad & \\
    x_n\sum_{j=1}^t \lambda_j \frac{\partial}{\partial x_n} (\tilde{f_j})  &= u_n.
\end{align*}

Multiplying $\tilde{f}_j$ by $\lambda_j$, this becomes

\begin{align*}
    \lambda_1 \tilde{f_1} = \ldots = \lambda_t \tilde{f_t} &= 0 \\
    x_1\sum_{j=1}^t \lambda_j \frac{\partial}{\partial x_1} (\tilde{f_j}) &= u_1 \\
    \vdots \qquad & \\
    x_n\sum_{j=1}^t \lambda_j \frac{\partial}{\partial x_n} (\tilde{f_j})  &= u_n.
\end{align*}

Observe that $\lambda_j \tilde{f}_j$ has the same monomial support as $x_i \lambda_j \frac{\partial}{\partial x_i} (\tilde{f}_j)$ for all $i \in [n]$. Therefore if we write $\lambda_j \tilde{f}_j = \sum_{i=1}^{M_j} c_{i,j} x^{\alpha_{i,j}}$ we can write $\init_w( \ell_1,\ldots, \ell_n,f_1,\ldots, f_t)$ as the linear system $AX = U$:
\begin{align*}
    \begin{bmatrix}
    \vline & & \vline  & & \vline & &\vline & &  \\
    \alpha_{1,1}  & \hdots & \alpha_{1,M_1}  & \hdots & \alpha_{t,1} & \hdots & \alpha_{t,M_t} \\
     \vline &  &  \vline   & &\vline & & \vline \\
   -  & \mathbf{1}_{M_1} & - &  & 0  & \cdots & 0   \\
    & \vdots &   & & \vdots & &\vdots \\
    0 & \cdots & 0 &  \cdots   & - & \mathbf{1}_{M_t} & -
    \end{bmatrix} \cdot \begin{bmatrix} c_{1,1} x^{\alpha_{1,1}} \\ \vdots \\ c_{1,M_1} x^{\alpha, M_1} \\ \vdots \\ c_{t,1} x^{\alpha_t,1} \\ \vdots \\ c_{t,M_t} x^{\alpha_{t, M_t}} \end{bmatrix} &= \begin{bmatrix} u_1 \\ \vdots \\ u_n \\ 0 \\ \vdots \\ 0 \end{bmatrix}.
\end{align*}

Note that $A \in \mathbb{N}^{(n+t) \times (M_1+\ldots + M_t)}$, $X \in \mathbb{R}^{M_1+\ldots + M_t}$,  
$U \in \mathbb{N}^{n+t}$, and 
$\mathbf{1}_{M_i}$ is a row vector of size $M_i$ of all ones.

For $M_1,\dots,M_t$ large enough, a dimension count of $A$ suggests its rows are linearly independent. However, it turns out that no matter the size of $M_1,\dots, M_t$, 
the matrix $A$ always has a nontrivial left kernel vector:
\[
    (a_1,\ldots, a_n, -\val_w(f_1),\ldots, - \val_w(f_t)),
\]
where $w = (a,b) \in \mathbb{Z}^n \times \mathbb{Z}^k$. 
This follows from
$\langle a, \alpha_{j,i}\rangle =\val_w(f_j) $
for $j\in[t]$.

By Lemma~\ref{lem:w_tilde_0}, we know that some of the $a_i$ are nonzero, which contradicts the generality of $u$, 
as it would imply that $\langle a, u \rangle = 0$ with $a \neq 0$.

\end{proof}

\subsection{Main result and consequences}

\begin{theorem}[Main Result]\label{theorem:main}
For general sparse polynomials $F= \langle f_1,\dots,f_k\rangle$, 
denote its
 ML system $\eqref{eq:MLsystem}$ by %
$\mathcal{L}(F)$.
The ML degree of $F$ equals the mixed volume of $\mathcal{L}(F)$. 
\end{theorem}
\begin{proof}
For readability, we abbreviate $\mvol(\newt(g_1),\dots,\newt(g_n))$ by  $\mvol(g_1,\dots,g_n)$ throughout the proof.
Consider the system $\hat{F} = (\hat{f}_1,\ldots, \hat{f}_k)$ where $\hat{f}_j= x_1\cdots x_n \cdot f_j$ for $j \in~[k]$. It follows from Theorem~\ref{thm:BKK}, Proposition~\ref{prop:finite_solutions} and Lemmas~\ref{lem:min_faces}, \ref{lem:case1}, ~\ref{lem:case2} and \ref{lem:case3} that the ML degree of $\hat{F}$ equals the mixed volume of $\mathcal{L}(\hat{F})$.

Now we show that the ML degree of $F$ equals the mixed volume of $\mathcal{L}(F)$. First observe that $\newt(f_j) + (\mathbf{1}_n, 0_k) = \newt(\hat{f}_j)$ for $j \in [k]$. Since the mixed volume is translation invariant, this gives
\[
\mvol(\ell_1,\ldots, \ell_n, f_1,\ldots, f_k)
=
\mvol(\ell_1,\ldots, \ell_n, \hat f_1,\ldots, \hat f_k).
\]

For $\phi\in \SL_{n+k}$ given by 
\[
\begin{bmatrix}
I_n & \mathbf{1}_{n\times k}\\
 0_{k\times n}   & I_k
\end{bmatrix},
\]
we have
\begin{align*}
\mvol(\ell_1,\ldots, \ell_n, \hat f_1,\ldots, \hat f_k)
&=
\mvol(\phi\cdot \ell_1,\ldots, \phi\cdot  \ell_n, \phi\cdot \hat f_1,\ldots, \phi\cdot \hat f_k) \\
&=
\mvol(\phi\cdot \ell_1,\ldots, \phi\cdot  \ell_n, \hat f_1,\ldots, \hat f_k).
\end{align*}

Since $\phi(\newt(\ell_i)) \subseteq \newt(\hat{\ell}_i) $,
by monotonicity of mixed volume we get 
\[
\mvol(\phi\cdot \ell_1,\ldots, \phi\cdot  \ell_n, \hat f_1,\ldots, \hat f_k)
\leq 
\mvol(\hat \ell_1,\ldots,  \hat \ell_n, \hat f_1,\ldots, \hat f_k).
\]

Thus far we have shown the inequality $\mvol(\mathcal{L}(F))\leq \mvol(\mathcal{L}(\hat F))$. We claim the following list of equalities also holds:
\begin{align}
\mvol(   \mathcal{L}(\hat{F})  ) \,\,=\,\, \text{ML Degree of } \hat{F} \,\,=\,\, \text{ML Degree of } F \,\,\leq \,\, \mvol(\mathcal{L}(F)). \label{eq:string_of_eqs}
\end{align}

From the argument given above we also have that the 
mixed volume of $\mathcal{L}(\hat{F})$ is equal to the
ML degree of $\hat{F}$. 
By Proposition~\ref{prop:Fhat_equals_F} we have that the ML degree of $\hat{F}$ equals the ML degree of $F$. 
The first part of Theorem~\ref{thm:BKK} tells us that the ML degree of $F$ is upper bounded by the mixed volume of $\mathcal{L}(F)$. The inequalitiy $\mvol(\mathcal{L}(F)) \leq \mvol(\mathcal{L}(\hat{F}))$ paired with \eqref{eq:string_of_eqs} shows that the mixed volume $\mathcal{L}(F)$ equals the ML degree of $F$. 
\end{proof}

\begin{rem}
Theorem~\ref{theorem:main} shows that an optimal homotopy method to find all critical points for maximum likelihood estimation is given by a standard polyhedral homotopy from its ML system. The polyhedral homotopy was presented in \cite{MR1297471}, and there exists off the shelf software implementations \cite{HomotopyContinuationjulia, verschelde1999algorithm, hom4ps2}. 
For the background in homotopy methods in
numerical algebraic geometry see~\cite{bertini-book}.
\end{rem}

\begin{rem}[Sum-to-one-constraint] %
For MLE, typically $f_1 \in F$ will be $x_1 + \ldots + x_n -1$. 
Although this polynomial does not have general coefficients, we can rescale the variables so the traditional MLE situation falls into our set-up. 
\end{rem}

A corollary of our results is that the ML degree of a general sparse polynomial system $F$ depends only on the Newton polytopes.

\begin{cor}
Consider two general sparse polynomial systems: 
$F = \langle f_1,\ldots, f_k\rangle$ and 
$G = \langle g_1,\ldots, g_k\rangle $, 
where $\newt(f_j) = \newt(g_j)$ for $j \in [k]$. 
The ML degree of $F$ equals the ML degree of $G$.
\end{cor}
\begin{proof}
Suppose $F$ and $G$ have the same Newton polytopes. Consider $x_1\cdots x_n F = \hat{F}$ and $x_1 \cdots x_n G = \hat{G}$. The ML systems of $\hat{F}$ and $\hat{G}$ have the same Newton polytopes, so by Theorem~\ref{theorem:main} the ML degree of $\hat{F}$ equals the ML degree of $\hat{G}$. Proposition~\ref{prop:Fhat_equals_F} then gives that the ML degree of $F$ equals the ML degree of $\hat{F}$ and likewise for $G$ and $\hat{G}$, giving the~result.
\end{proof}

This is a surprising corollary because the Newton polytopes of $F$ 
do \emph{not} determine the Newton polytopes of the respective ML system.

\begin{ex}
Consider the ML systems $\mathcal{L}(f)$ and $\mathcal{L}(\hat{f})$ from Example~\ref{ex:newtFvsFhat} and Example~\ref{ex:min_faces}. Both of these systems have a mixed volume and ML degree equal to $12$ even though the Newton polytopes of the corresponding 
ML systems
are quite different. 
\end{ex}

\bibliographystyle{abbrv}
\bibliography{refs.bib}    

\begin{thebibliography}{10}

\bibitem{aluffiEDComplex}
P.~Aluffi and C.~Harris.
\newblock The {E}uclidean distance degree of smooth complex projective
  varieties.
\newblock {\em Algebra Number Theory}, 12(8):2005--2032, 2018.

\bibitem{MR3907355}
C.~Am\'{e}ndola, N.~Bliss, I.~Burke, C.~R. Gibbons, M.~Helmer, S.~Ho\c{s}ten,
  E.~D. Nash, J.~I. Rodriguez, and D.~Smolkin.
\newblock The maximum likelihood degree of toric varieties.
\newblock {\em J. Symbolic Comput.}, 92:222--242, 2019.

\bibitem{AKRS2021}
C.~Am\'{e}ndola, K.~Kohn, P.~Reichenbach, and A.~Seigal.
\newblock Invariant theory and scaling algorithms for maximum likelihood
  estimation.
\newblock {\em SIAM J. Appl. Algebra Geom.}, 5(2):304--337, 2021.

\bibitem{baaijensRealAlgGroups}
J.~A. Baaijens and J.~Draisma.
\newblock Euclidean distance degrees of real algebraic groups.
\newblock {\em Linear Algebra Appl.}, 467:174--187, 2015.

\bibitem{bertini-book}
D.~J. Bates, J.~D. Hauenstein, A.~J. Sommese, and C.~W. Wampler.
\newblock {\em Numerically solving polynomial systems with {B}ertini},
  volume~25 of {\em Software, Environments, and Tools}.
\newblock Society for Industrial and Applied Mathematics (SIAM), Philadelphia,
  PA, 2013.

\bibitem{bernshtein1979the}
D.~N. Bernstein.
\newblock The number of roots of a system of equations.
\newblock {\em Funkcional. Anal. i Prilo\v{z}en.}, 9(3):1--4, 1975.

\bibitem{bihan2019criteria}
F.~Bihan and I.~Soprunov.
\newblock Criteria for strict monotonicity of the mixed volume of convex
  polytopes.
\newblock {\em Adv. Geom.}, 19(4):527--540, 2019.

\bibitem{blekhermanConvexAGBook}
G.~Blekherman, P.~A. Parrilo, and R.~R. Thomas, editors.
\newblock {\em Semidefinite optimization and convex algebraic geometry},
  volume~13 of {\em MOS-SIAM Series on Optimization}.
\newblock Society for Industrial and Applied Mathematics (SIAM), Philadelphia,
  PA; Mathematical Optimization Society, Philadelphia, PA, 2013.

\bibitem{breiding2020euclidean}
P.~{Breiding}, F.~{Sottile}, and J.~{Woodcock}.
\newblock Euclidean distance degree and mixed volume.
\newblock {\em arXiv preprint arXiv:2012.06350}, 2020.

\bibitem{HomotopyContinuationjulia}
P.~Breiding and S.~Timme.
\newblock Homotopycontinuation.jl: A package for homotopy continuation in
  {J}ulia.
\newblock In {\em Mathematical Software -- ICMS 2018}, pages 458--465. Springer
  International Publishing, 2018.

\bibitem{SDSS}
T.~Brysiewicz, J.~I. Rodriguez, F.~Sottile, and T.~Yahl.
\newblock Solving decomposable sparse systems.
\newblock {\em Numer. Algorithms}, 88(1):453--474, 2021.

\bibitem{MR2230921}
F.~Catanese, S.~Ho\c{s}ten, A.~Khetan, and B.~Sturmfels.
\newblock The maximum likelihood degree.
\newblock {\em Amer. J. Math.}, 128(3):671--697, 2006.

\bibitem{MR4115654}
D.~Cifuentes, C.~Harris, and B.~Sturmfels.
\newblock The geometry of {SDP}-exactness in quadratic optimization.
\newblock {\em Math. Program.}, 182(1-2, Ser. A):399--428, 2020.

\bibitem{MR4103774}
P.~Clarke and D.~A. Cox.
\newblock Moment maps, strict linear precision, and maximum likelihood degree
  one.
\newblock {\em Adv. Math.}, 370:107233, 51, 2020.

\bibitem{idealscox}
D.~A. Cox, J.~Little, and D.~O'Shea.
\newblock {\em Ideals, varieties, and algorithms}.
\newblock Undergraduate Texts in Mathematics. Springer, Cham, fourth edition,
  2015.
\newblock An introduction to computational algebraic geometry and commutative
  algebra.

\bibitem{DM2021}
H.~Derksen and V.~Makam.
\newblock Maximum likelihood estimation for matrix normal models via quiver
  representations.
\newblock {\em SIAM J. Appl. Algebra Geom.}, 5(2):338--365, 2021.

\bibitem{DraismaTheEDD}
J.~Draisma, E.~Horobe\c{t}, G.~Ottaviani, B.~Sturmfels, and R.~Thomas.
\newblock The {E}uclidean distance degree.
\newblock In {\em S{NC} 2014---{P}roceedings of the 2014 {S}ymposium on
  {S}ymbolic-{N}umeric {C}omputation}, pages 9--16. ACM, New York, 2014.

\bibitem{MR3451425}
J.~Draisma, E.~Horobe\c{t}, G.~Ottaviani, B.~Sturmfels, and R.~R. Thomas.
\newblock The {E}uclidean distance degree of an algebraic variety.
\newblock {\em Found. Comput. Math.}, 16(1):99--149, 2016.

\bibitem{drusvyatskiyOrthogonally}
D.~Drusvyatskiy, H.-L. Lee, G.~Ottaviani, and R.~R. Thomas.
\newblock The {E}uclidean distance degree of orthogonally invariant matrix
  varieties.
\newblock {\em Israel J. Math.}, 221(1):291--316, 2017.

\bibitem{mvoltext}
G.~Ewald.
\newblock {\em Combinatorial convexity and algebraic geometry}, volume 168 of
  {\em Graduate Texts in Mathematics}.
\newblock Springer-Verlag, New York, 1996.

\bibitem{MR2496496}
H.-C. Graf~von Bothmer and K.~Ranestad.
\newblock A general formula for the algebraic degree in semidefinite
  programming.
\newblock {\em Bull. Lond. Math. Soc.}, 41(2):193--197, 2009.

\bibitem{MR2988436}
E.~Gross, M.~Drton, and S.~Petrovi\'{c}.
\newblock Maximum likelihood degree of variance component models.
\newblock {\em Electron. J. Stat.}, 6:993--1016, 2012.

\bibitem{MR0463157}
R.~Hartshorne.
\newblock {\em Algebraic geometry}.
\newblock Springer-Verlag, New York-Heidelberg, 1977.
\newblock Graduate Texts in Mathematics, No. 52.

\bibitem{HostenSolving}
S.~Ho\c{s}ten, A.~Khetan, and B.~Sturmfels.
\newblock Solving the likelihood equations.
\newblock {\em Found. Comput. Math.}, 5(4):389--407, 2005.

\bibitem{huber1995a}
B.~Huber and B.~Sturmfels.
\newblock A polyhedral method for solving sparse polynomial systems.
\newblock {\em Math. Comp.}, 64(212):1541--1555, 1995.

\bibitem{MR1297471}
B.~Huber and B.~Sturmfels.
\newblock A polyhedral method for solving sparse polynomial systems.
\newblock {\em Math. Comp.}, 64(212):1541--1555, 1995.

\bibitem{MR3103064}
J.~Huh.
\newblock The maximum likelihood degree of a very affine variety.
\newblock {\em Compos. Math.}, 149(8):1245--1266, 2013.

\bibitem{khovanskii1978newton}
A.~G. Khovanskii.
\newblock Newton polyhedra, and the genus of complete intersections.
\newblock {\em Funktsional. Anal. i Prilozhen.}, 12(1):51--61, 1978.

\bibitem{kouchnirenko1976polyedres}
A.~G. Kouchnirenko.
\newblock Poly\`edres de {N}ewton et nombres de {M}ilnor.
\newblock {\em Invent. Math.}, 32(1):1--31, 1976.

\bibitem{leeFermat}
H.~Lee.
\newblock The {E}uclidean distance degree of {F}ermat hypersurfaces.
\newblock {\em J. Symbolic Comput.}, 80(part 2):502--510, 2017.

\bibitem{hom4ps2}
T.~L. Lee, T.~Y. Li, and C.~H. Tsai.
\newblock H{OM}4{PS}-2.0: a software package for solving polynomial systems by
  the polyhedral homotopy continuation method.
\newblock {\em Computing}, 83(2-3):109--133, 2008.

\bibitem{maximDefect}
L.~G. Maxim, J.~I. Rodriguez, and B.~Wang.
\newblock Defect of {E}uclidean distance degree.
\newblock {\em Adv. in Appl. Math.}, 121:102101, 22, 2020.

\bibitem{maximMultiview}
L.~G. Maxim, J.~I. Rodriguez, and B.~Wang.
\newblock Euclidean distance degree of the multiview variety.
\newblock {\em SIAM J. Appl. Algebra Geom.}, 4(1):28--48, 2020.

\bibitem{MR4219257}
M.~Micha\l~ek, L.~Monin, and J.~a.~A. Wi\'{s}niewski.
\newblock Maximum likelihood degree, complete quadrics, and
  {$\mathbb{C}^*$}-action.
\newblock {\em SIAM J. Appl. Algebra Geom.}, 5(1):60--85, 2021.

\bibitem{MR2507133}
J.~Nie and K.~Ranestad.
\newblock Algebraic degree of polynomial optimization.
\newblock {\em SIAM J. Optim.}, 20(1):485--502, 2009.

\bibitem{MR2546336}
J.~Nie, K.~Ranestad, and B.~Sturmfels.
\newblock The algebraic degree of semidefinite programming.
\newblock {\em Math. Program.}, 122(2, Ser. A):379--405, 2010.

\bibitem{MR3686780}
J.~I. Rodriguez and B.~Wang.
\newblock The maximum likelihood degree of mixtures of independence models.
\newblock {\em SIAM J. Appl. Algebra Geom.}, 1(1):484--506, 2017.

\bibitem{Sturmfels-CBMS}
B.~Sturmfels.
\newblock {\em Solving systems of polynomial equations}, volume~97 of {\em CBMS
  Regional Conference Series in Mathematics}.
\newblock Published for the Conference Board of the Mathematical Sciences,
  Washington, DC; by the American Mathematical Society, Providence, RI, 2002.

\bibitem{MR4196404}
B.~Sturmfels, S.~Timme, and P.~Zwiernik.
\newblock Estimating linear covariance models with numerical nonlinear algebra.
\newblock {\em Algebr. Stat.}, 11(1):31--52, 2020.

\bibitem{verschelde1999algorithm}
J.~{Verschelde}.
\newblock Algorithm 795: Phcpack: a general-purpose solver for polynomial
  systems by homotopy continuation.
\newblock {\em ACM Transactions on Mathematical Software}, 25(2):251--276,
  1999.

\bibitem{VVC}
J.~Verschelde, P.~Verlinden, and R.~Cools.
\newblock Homotopies exploiting {N}ewton polytopes for solving sparse
  polynomial systems.
\newblock {\em SIAM J. Numer. Anal.}, 31(3):915--930, 1994.

\end{thebibliography}
\end{document}